\documentclass[10pt]{amsart}
\usepackage{hyperref}
\usepackage{amsmath,amsfonts,mathrsfs,amssymb,latexsym,array,mathtools}
\usepackage[all]{xy}
\newcommand{\N}{\mathcal{S}} 
\newcommand{\M}{{\mathcal{M}_{g,1}^{\N}}} 
\newcommand{\C}{\mathcal{C}} 
\newcommand{\D}{\mathcal{D}} 
\newcommand{\LO}{\mathcal{O}} 
\renewcommand{\k}{\mathbf{k}} 
\newcommand{\proj}{\mathbb{P}} 
\renewcommand{\l}{\ell} 
\newcommand{\F}[1]{F_{#1}} 
\usepackage{color}
\renewcommand{\sharp}{\#}

\DeclareMathOperator*{\codim}{codim} 
\DeclareMathOperator*{\Spec}{Spec} %
\DeclareMathOperator*{\ewt}{ewt}

\DeclareMathOperator*{\End}{End}
\newtheorem{question}{Question}[section]

\newtheorem{thm}{Theorem}[section]
\newtheorem{lem}[thm]{Lemma}
\newtheorem{cor}[thm]{Corollary}

\newtheorem{rmk}[thm]{Remark}
\newtheorem*{syzlem*}{Syzygy Lemma}
\def\c[#1,#2]{{c_{{#1},#2}}}
\def \f[#1,#2]{{f_{#1}^{(#2)}}}
\def \g[#1,#2]{{g_{#1}^{(#2)}}}



\begin{document}

\title[On the dimension of $\M$]{On the dimension of the stratum of the moduli of pointed curves by Weierstrass gaps}

\author{Andr\'{e} Contiero}
\address{Departamento de Matem\'atica, ICEx, UFMG. Av. Ant\^onio Carlos 6627, 30123-970 Belo Horizonte MG, Brazil}
\email{contiero@ufmg.br}

\author{Aislan Leal Fontes}
\address{Departamento de Matem\'atica, UFS - Campus Itabaiana. Av. Vereador Ol\'impio Grande s/n, 49506-036 Itabaiana/SE, Brazil}
\email{aislan@ufs.br}

\author{Jhon Quispe Vargas}
\address{Departamento de Matem\'atica, ICEx, UFMG. Av. Ant\^onio Carlos 6627, 30123-970 Belo Horizonte MG, Brazil}
\email{jhon.quispev@gmail.com}

\begin{abstract}
The dimension of the moduli space of smooth pointed curves with
prescribed Weierstrass semigroup at the market point is computed
for three families of symmetric semigroups of multiplicity six.
We also 
collect the dimensions of such moduli spaces for all
semigroups of genus not greater than seven. A question related to an improvement of a
Deligne--Pinkham's bound is also formulated, suggesting that the positive graded part of the 
first module of the cotangent complex associated to
a semigroup algebra is a missing invariant. 
The answer for this question is positive for all 
these moduli that we know their dimensions.
\end{abstract}

\maketitle

\section{Introduction}

Given a smooth pointed curve $(\C,P)\in\mathcal{M}_{g,1}$ of genus $g>1$, its associated Weiers\-trass semigroup $\N$ 
consists of the set of nonnegative integers $n$, called nongaps, such that there is a rational function on $\C$ with
pole divisor $nP$. Equivalently, $n$ is a nongap if and only if 
$\mathrm{H}^0(\C, \LO_{C}(n-1)P)\subsetneq\mathrm{H}^{0}(\C, \LO_{\C}(nP))$. Follows from the Riemann--Roch theorem 
that the set of positive integers that are not in $\N$ (the set of gaps) has magnitude exactly $g$.

Inverting the above considerations, given a numerical semigroup $\N$ of genus $g>1$, let $\M$ be the moduli 
space parameterizing smooth pointed curves whose associa\-ted Weierstrass semigroup is $\N$. It is very known that 
the moduli space $\M$ can be empty. If it is nonempty, since the $i$-th gap of a Weierstrass
semigroup is an upper semicontinuous function, the moduli space $\M$ is a locally closed subspace of $\mathcal{M}_{g,1}$.
Hence, the moduli spaces $\M$ give a stratification of $\mathcal{M}_{g,1}$.
Naturally, many general problems about this stratification arise. But, the simplest problems of determine when $\M$ is nonempty, 
describe its irreducible components and their dimensions, remain unsolved. 

In this paper we focus on the problem  of find the dimension of a given $\M$. There are two general and important
bounds for the dimension of $\M$. On the one hand a Deligne--Pinkham's upper
bound, cf. \cite{Del73, Pi74}, and on the other hand a lower bound given by Pflueger in \cite{N16}. The section two of the 
present paper details this two bounds, it is also showed that this two bounds can be not attained.

In section three of this paper we recall an important construction given by Stoehr \cite{St93} and Contiero--Stoehr \cite{c2013}
 of a compactification of $\M$, when $\N$ is symmetric, by allowing canonical Gorenstein curves at its bordering. 
 
 In section four  we compute the dimension of $\M$ when $\N$ runs over the following two families of symmetric semigroups, 
 $\N:=<6,3+6\tau,4+6\tau,7+6\tau,8+6\tau>$ and $\N:=<6,7+6\tau,8+6\tau,9+6\tau,10+6\tau>$, with $\tau>0$. The method to compute the dimension
 is to construct a compactification of $\M$ as a quotient of an affine quasi-cone $\mathcal{X}$ by an action of the multiplicative group
 $\mathbb{G}_{m}(\k)$, this approach gives a compactification of the moduli space $\M$ constructed by Pinkham in \cite{Pi74}
 using equivariant deformation theory. Then we realize a quadratic approximation of $\mathcal{X}$ as a scheme over a
 suitable artinian algebra. Hence we get an upper bound for the dimension of $\M$ which is better than Deligne--Pinkham's
 bound and it is equal to Pflueger's lower bound.
 
 Finally, in section five we collect the dimensions of $\M$ for all semigroups of genus not greater than seven, and also
 for three families of symmetric semigroups. Therefore, we can verify that for all these semigroups the dimension
 of $\M$ is equal to Pflueger's lower bound and also it is equal to 
 \begin{equation}\label{chute}
 2g-2+\lambda(\N)-\dim\mathrm{T}^{1,+}(\k[\N]),
 \end{equation} where
 $\dim\mathrm{T}^{1,+}(\k[\N])$ stands for the dimension of the positive graded part of the first module of the cotangent complex associated
 to the semigroup algebra $\k[\N]$ introduced by Lichtenbaum and Schlessinger in \cite{LS}. We also noted that there are examples where 
 the Pflueger's bound does not provide the exact dimension of $\M$, but the above possible bound \eqref{chute}  does, for example $\N=<6,7,8>$ and $\N=<6,7,15>$. 
 As a final result we show that Pflueger's bound is not greater than $2g-2+\lambda(\N)-\dim\mathrm{T}^{1,+}(\k[\N])$ for any numerical
 semigroup $\N$ of genus $g>0$.

\section{General bounds for $\dim\M$}

Let $\C$ be a projective reduced algebraic curve defined over an algebraically closed field $\k$ and $P\in C$ a rational point. 
Set $\delta:=\dim_{\k} \overline{\LO}/{\LO}$ the singularity degree of $P$ and 
$\mu:=\dim(\mathrm{Der}_{\k}(\overline{\LO},\overline{\LO})/\mathrm{Der}_{\k}(\LO,\LO))$, where
$\mathrm{Der}_{\k}(R,M)$ is the module of $\k$-derivations from $R$ to $M$.
By analyzing three different parameters spaces, using a result of Rim \cite[Cor. 2.10]{Rim} and the
Kodaira--Spencer correspondence, Deligne \cite[Thm 2.27]{Del73} established the following formula.

\begin{thm}[Deligne's formula]\label{delformula}
For any smoothing component $E$ of the formal versal space of deformations of $\LO$, it follows
$$\dim E=3\delta-\mu\,.$$
\end{thm} A smoothing component stands for an irreducible component whose fiber over its generic
point does not intersect the singular locus of the total space.

An interesting consequence of Deligne's theorem can be obtained by assuming that $\C$ is a monomial curve. In this case,
by computing the right hand-side of Deligne's formula, do not assuming the smoothing condition, using
Pikham's equivariant deformation theory \cite[Thm. 13.9]{Pi74},  Rim and Vitulli \cite[\S 6]{RV77} noted
that:

\begin{thm}[Deligne--Pinkham's bound]\label{delupper} For any numerical semigroup $\N$,
 $$\dim \M\leq 2g-2+\lambda(\N)=3\delta-\mu$$
where $\lambda(\N)$ is the number of gaps $\l$ such that $\l+n\in\N$ whenever $n$ is a nongap.
\end{thm}

Deligne--Pinkham's upper bound is attained, Rim and Vitulli showed, cf. \cite[Cor. 5.14]{RV77}, that if $\N$ is 
\textit{negatively graded}, then $\dim \M=2g-2+\lambda(\N)$. A numerical semigroup is negatively
graded if the first cohomology module of the cotangent complex associated to the semigroup algebra
$\k[\N]$ is negatively graded. Despite this, there are families of semigroups where Deligne--Pinkham's bound is far from
being tight, see sections 4 and 5 below, specially the tables \ref{tab1} and \ref{tab2} in section 5 of this paper.

On the other hand, Pflueger in \cite{N16} produced an upper bound for the codimension
of $\M$ seen it as a locally closed subset of $\mathcal{M}_{g,1}$. His bound
is an improvement of a bound given by Eisenbud and Harris in \cite{EH87}. Pflueger introduced 
the \textit{effective weight} of a numerical semigroup $\N$,
$$\ewt(\N):=\displaystyle\sum_{\mbox{gaps}\ l_i}^{}(\#\ \mbox{generators}\ n_j<l_i),$$
as a substitute of the weight of a numerical semigroup which appears in Eisenbud--Harris bound.
More specifically, Pflueger showed \cite[Thm. 1.2]{N16}:

\begin{thm}(Pflueger's bound)\label{NPbound}
If $\mathcal{M}_{g,1}^{\N}$ is nonempty, and $X$ is any irreducible component of it, then
\begin{align*}
\dim X\geq 3g-2-\ewt(\N).
\end{align*}
\end{thm} 

Although Pflueger's bound is attained for a lot of classes of numerical semigroups, 
there are examples where his bound does not give the exact dimension
of $\M$, an example was given by Pflueger himself, cf. \cite[pg. 12]{N16}, taking the follow\-ing symmetric semigroup $\N:=<6,7,8>$
of genus $9$.
The moduli variety for this particular semigroup can be completely described by classical results as follows.
Let $\N$ be a symmetric semigroup generated by less than $5$ elements. Using Pinkham's equivariant deformation theory \cite{Pi74},
a quasi-homogeneous version of 
Buchsbaum-Eisenbud's structure theorem
for Gorenstein ideals of codimension $3$ (see
\cite[p.\,466]{BE77}), one can deduce that
the affine monomial curve 
$\Spec\ \k[\N] = \Spec\ \k[t^{m_1},t^{m_2},t^{m_3},t^{m_4}]$
can be negatively smoothed without any obstructions (see \cite{Bu80},
\cite{W79} \cite[Satz 7.1]{W80}), hence
$\dim \M = \dim \proj(T^{1,-}_{\k[\N]|\k})$, and therefore, by the Jacobian criterion and elimination theory,
\begin{equation}\label{truco}
\overline{\M} = \proj(T^{1,-}_{\k[\N]|\k}).
\end{equation} If $\N=<6,7,8>$, then $\dim\M=14$, see theorem \ref{t1busc} in section 5 below. In contrast,
Pflueger's bound gives $3g-2-\mathrm{ewt}(\N)=27-2-12=13$, while Deligne--Pinkham's
bound provides $2g-2+\lambda(\N)=18-2+1=17$.  In the same way, if we take the 
following symmetric semigroup $\N:=<6,7,15>$ of genus $12$, Pflueger's bound gives
$17$, while Deligne--Pinkham's bound $23$. But one can verify using \eqref{truco} and theorem \ref{t1busc} that 
$\dim\M=\dim\mathbb{P}(T^{1,-}(\k[\N]))=18$.

In a later paper, cf. \cite{NP2}, Pflueger made a detailed study of the moduli variety 
$\M$ when $\N$ is a \textit{Castelnuovo semigroup}, ie. semigroups generated
by consecutive suita\-ble positive integers. In \cite{N16} Pflueger noted that this class 
of semigroups is one he is aware that his bound in \ref{NPbound} does not provide the exact dimension of
$\M$, see \cite[pg. 2]{N16}. 

Summarizing, for any semigroup $\N$ such that $\M\neq\emptyset$ we have


\begin{equation*}\label{2bounds}
3g-2-\mathrm{ewt}(\N)\leq\dim\M\leq 2g-2+\lambda(\N).
\end{equation*}

\section{Weierstrass points on Gorenstein curves}
Let $\C$ be a complete integral Gorenstein curve of arithmetic genus $g>1$ defined
over an algebraically closed field $\k$. For each smooth point $P$ of $\C$, let 
$\N$ be the Weierstrass semigroup of $\C$ at $P$.
By the very definition, for each $n\in\N$ 
there is a rational function $x_n$ on $\C$ with pole divisor $nP$. Let us assume 
that the semigroup $\N$ is symmetric, ie. the last gap $\l_g$ is the biggest possible, $l_g=2g-1$.
Equivalently, $n\in\N$ if, and only if, $\l_g-n\notin\N$.
Let $\omega$ be 
the dualizing sheaf of $\C$. A basis for the vector space $H^{0}(\C, \omega)$ is
$\{x_{n_0}, x_{n_1},\ldots, x_{n_{g-1}}\}$, 
and thus $\omega\cong\mathcal{O}_{\C}((2g-2)P)$. By assuming that $\C$ is nonhyperelliptic, 
the canonical morphism
\begin{align*}
(x_{n_0}: x_{n_1}:\ldots: x_{n_{g-1}}):\C\hookrightarrow\mathbb{P}^{g-1}
\end{align*} is an embedding.
Thus $\C$ becomes a curve of degree $2g-2$ in $\mathbb{P}^{g-1}$ and the 
integers $l_i-1$ are the contact orders of the curve with the 
hyperplanes at $P=(0:\ldots:0:1)$. Conversely, any nonhyperelliptic symmetric
semigroup $\N$ can be realized as the Weierstrass semigroup of the Gorenstein
canonical \textit{monomial curve}
\begin{align*}
\C_{\N}:=\{(s^{n_0}t^{\l_g-1}: s^{n_1}t^{\l_{g-1}-1}:\ldots:s^{n_{g-2}}t^{\l_2-1}: s^{n_{g-1}}t^{\l_1-1})\,\vert\,(s:t)\in\mathbb{P}^1\}\subset \mathbb{P}^{g-1}\,
\end{align*}
at its unique point $P$ at the infinity.

Now we recall a construction of a compactification of $\M$ given by Stoehr \cite{St93} and Contiero--Stoehr \cite{c2013}. 
This construction will be required for the next section.

Since $\N$ is symmetric, each nongap $s\in\N, s\leq4g-4$ can be written as a sum of two others nongaps (see \cite[theorem 1.3]{O91}),
\begin{align*}
s=a_s+b_s,\ a_s\leq b_s\leq2g-2.
\end{align*}
By taking $a_s$ as the smallest possible, the $3g-3$ rational functions $x_{a_s}x_{b_s}$ form a $P$-hermitian 
basis of the space of global sections $H^0(\C,\omega^2)$ of the 
bicanonical divisor. If $r\geq3$, then a $P$-hermitian basis of 
the vector space $H^0(\C,\omega^r)$ (cf. \cite[Lemma 2.1]{c2013}) is
\begin{equation*}
\begin{array}{lr}
x_{n_{0}}^{r-1}x_{n_i} & (i=0,\ldots,g-1),\\
x_{n_{0}}^{r-2-i}x_{a_s}x_{b_s}x_{n_{g-1}}^{i} & (i=0,\ldots, r-2,\ s=2g,\ldots,4g-4),\\
x_{n_{0}}^{r-3-i}x_{n_1}x_{2g-n_1}x_{n_{g-2}}x_{n_{g-1}}^{i} & (i=0,\ldots, r-3).
\end{array} 
\end{equation*}
A consequence of the existence of a $P$-hermitian basis of $H^0(\C,\omega^r)$ for any $r\geq 1$, 
is a Max-Noether's theorem, namely the following homomorphism
\begin{equation*}
{\mathbf{k}[X_{n_0},\ldots,X_{n_{g-1}}]}_r\longrightarrow H^{0}(\C, \omega^r)
\end{equation*}
induced by the substitutions $X_{n_{i}}\longmapsto x_{n_i}$ is surjective for 
each $r\geq 1$, where ${\mathbf{k}[X_{n_0},\ldots,X_{n_{g-1}}]}_r$ is the vector space of $r$-forms.
  
Let $I(\C)=\bigoplus_{r=2}^{\infty}I_r(\C)\subset\mathbf{k}[X_{n_0},\ldots,X_{n_{g-1}}]$ 
be the ideal of $\C\subset\proj^{g-1}$. The codimension of $I_r(\C)$ in 
${\mathbf{k}[X_{n_0},\ldots,X_{n_{g-1}}]}_r$ is equal to $(2r-1)(g-1)$, 
in particular $$\dim I_2(\C)=(g-2)(g-3)/2.$$ For $r\geq2$, 
let  $\Lambda_r$ be the vector space in ${\mathbf{k}[X_{n_0},\ldots,X_{n_{g-1}}]}_r$ 
spanned by the lifting of the $P$-hermitian basis of $H^{0}(\C, \omega^r)$. 
Since $\Lambda_r\cap I_r(\C)=0$ and
\begin{align*}
\dim\Lambda_r=\dim H^{0}(\C, \omega^r)=\codim I_r(\C),
\end{align*}
it follows that
\begin{align*}
{\mathbf{k}[X_{n_0},\ldots,X_{n_{g-1}}]}_r=\Lambda_r\oplus I_r(\C),\ r\geq2.
\end{align*}
For each nongap $s\leq4g-4$, let us consider all the partitions of $s$ as sum of 
two nongaps not greater than $2g-2$,
\begin{align*}
s=a_{si}+b_{si},\ \mbox{with}\ a_{si}\leq b_{si},\ (i=1,\ldots, \nu_s), \mbox{ where } a_{s0}:=a_s.
\end{align*}
Hence, given a nongap $s\leq4g-4$ and $i=1,\ldots,\nu_s$  we can write
\begin{align*}
x_{a_{si}}x_{b_{si}}=\displaystyle\sum_{n=0}^{s}c_{sin}x_{a_{n}}x_{b_{n}},
\end{align*}
where $a_n$ and $b_n$ are nongaps of $\N$ whose sum is equal to $s$, and $c_{sin}$ are suitable constants in $\k$.
By normalizing 
the coefficients $c_{sis}=1$, it follows that the 
${g+1\choose 2}-(3g-3)=\frac{(g-2)(g-3)}{2}$ quadratic forms 
\begin{align*}
F_{si}=X_{a_{si}}X_{b_{si}}-X_{a_s}X_{b_s}-\displaystyle\sum_{n=0}^{s-1}c_{sin}X_{a_{n}}X_{b_{n}}
\end{align*} 
vanish identically on the canonical curve $\C$, where the coefficients 
$c_{sin}$ are uniquely determined constants. They are linearly independent, 
hence they form a basis for the space of quadratic relations $I_2(\C)$.
 
It is necessary to make some assumptions on the symmetric semigroup $\N$ to assure 
that the ideal $I(\C)$ is generated by quadratic relations. More 
precisely, we suppose that $\N$ satisfies $3<n_1<g$ and $\N\neq\langle4, 5\rangle$. 
According to \cite[Lemma 3.1]{CF}, both the conditions $n_1\neq3$ and $n_1\neq g$ 
on $\N$ are to avoid possible trigonal Gorenstein curves whose Weierstrass semigroup at $P$ 
equal to $\N=\langle3, g+1\rangle$ and 
$\N=\langle g, g+1,\ldots, 2g-2\rangle$, respectively. This two avoided cases are treated
by Contiero and Fontes in \cite{CF}.
By the assumptions on the semigroup $\N$ it follows by the Enriques--Babbage 
theorem  that  $\C$ is nontrigonal and it is not isomorphic to a plane quintic. 

If $\C$ is smooth, then Petri's theorem \cite{ACGH85} assure that the ideal of $\C$ is 
generated by the quadratic relations. Given a canonical curve $\C$, not necessarily 
smooth, an algorithmic proof that the ideal of $\C$ is generated by the quadratic 
forms $F_{si}$ was done by Contiero and Stoehr in \cite[Theorem 2.5]{c2013}.

On the other hand, for each symmetric semigroup $\N$ with $3<n_1<g$ and $\N\neq\langle4, 5\rangle$,
we can take the following $(g-2)(g-3)/2$ quadratic forms \begin{align*}
F_{si}=X_{a_{si}}X_{b_{si}}-X_{a_s}X_{b_s}-\displaystyle\sum_{n=0}^{s-1}c_{sin}X_{a_{n}}X_{b_{n}},
\end{align*} where $c_{sin}$ are constants to be determined in order that the intersection
$\cap V(F_{si})\subset\proj^{g-1}$ is a canonical Gorenstein curve of genus $g$ whose Weierstrass semigroup at $P$ is 
$\N$.
Analogously, let \begin{align*}
F_{si}^{(0)}:=X_{a_{si}}X_{b_{si}}-X_{a_s}X_{b_s}
\end{align*} be the quadratic forms that generate the ideal of the canonical monomial curve
$\C_{\N}$, cf. \cite[Lemma 2.2]{c2013}. One of the keys to construct a compactification of $\M$ is the following lemma.

\begin{syzlem*}[cf. \cite{c2013}]\label{syzlem}
For each of the $\frac{1}{2}(g-2)(g-5)$ quadratic binomials $\F{s'i'}^{(0)}$
different from 
$\F{n_i+2g-2,1}^{(0)}$ $(i=0,\dots,g-3)$ there is
a syzygy of the form
$$X_{2g-2}\F{s'i'}^{(0)}+\sum_{nsi}\varepsilon_{nsi}^{(s'i')}X_{n}\F{si}^{(0)}=0$$
where the coefficients $\varepsilon_{nsi}^{(s'i')}$ are integers equal to $1$,
$-1$ or $0$, and where the sum is
taken over the
nongaps $n<2g-2$ and the double indexes $si$ with $n+s=2g-2+s'$.
\end{syzlem*}

Let us described briefly the algorithmic construction of a compactification of
$\M$ which was done by Stoehr \cite{St93} and Contiero--Stoehr \cite{c2013}.
First, we replace the binomials $\F{s'i'}^{(0)}$ and $\F{si}^{(0)}$ on the left hand 
side of the Syzygy Lemma by the corresponding quadratic forms $\F{s'i'}$ and
$\F{si}$. Hence we obtain a 
linear combination of cubic monomials of weight $<s'+2g-2$. By virtue of \cite[Lemma 2.4]{c2013} 
this linear combination of cubic monomials admits the following decomposition.
\begin{equation*}\label{Rsi}
 X_{2g-2}\F{s'i'}+\sum_{nsi}\varepsilon_{nsi}^{(s'i')}X_{n}\F{si}=
 \sum_{nsi}\eta_{nsi}^{(s'i')}X_{n}\F{si}+R_{s'i'}
\end{equation*}
where the sum on the right hand side is taken over the nongaps $n\leq 2g-2$ and
the double indexes $si$ with $n+s<s'+2g-2$, where the coefficients
$\eta_{nsi}^{(s'i')}$ are constants, and where $R_{s'i'}$ is a linear
combination of cubic monomials of pairwise different weights $<s'+2g-2$.

For each nongap $m<s'+2g-2$ we denote by $\varrho_{s'i'm}$ the unique coefficient
of $R_{s'i'}$ of weight $m$. Finally, let us consider the following
quasi-homogeneous polynomial in the constants $c_{sin}$,

\begin{equation*}
R_{s'i'}(t^{n_0},t^{n_1},\dots,t^{n_{g-1}})=
\sideset{}{'}\sum_{m=0}^{s'+2g-3}\varrho_{s'i'm}t^m \ .
\end{equation*}
\noindent


Since that the coordinates functions $x_n$, $n\in\N$ 
and $n\leq2g-2$, are not uniquely determined by their pole divisor $nP$
by assuming the characteristic of the field $\k$ to be zero (or a prime not 
dividing any of the differences $m-n$ where $n, m$ are nongaps of $\N$ such 
that $n, m\leq2g-2$), we transform
\begin{align*}
X_{n_i}\longmapsto X_{n_i}+\displaystyle\sum_{j=0}^{i-1}c_{n_in_{i-j}}X_{n_{i-j}},
\end{align*} 
for each $i=1,\ldots, g-1$, and so we can normalize 
$\frac{1}{2}g(g-1)$ of the coefficients $c_{sin}$ to be zero, see 
\cite[Proposition 3.1]{St93}. Due to these normalizations and the 
normalizations of the coefficients $c_{sin}=1$ with $n=s$, the only left to us 
is to transform $x_{n_i}\mapsto c^{n_i}x_{n_i}$ for $i=1,\ldots, g-1$. Summarizing, we get

\begin{thm}\cite[Theorem 2.6]{c2013}\label{teo3}
Let $\N$ be a symmetric semigroup of genus $g$ sa\-tis\-fying $3<n_1<g$ and 
$\N\neq\langle4, 5\rangle$.
The isomorphism classes of the pointed complete integral Gorenstein curves with Weierstrass 
semigroup $\N$ correspond bijectively to the orbits of the $\mathbb{G}_m(\textit{k})$-action
\begin{equation*}
(c,\ldots,c_{sin},\ldots)\longmapsto(\ldots,c^{s-n}c_{sin},\ldots)
\end{equation*}
on the  affine quasi-cone of the vectors whose coordinates are the 
coefficients $c_{sin}$ of the normalized quadratic $F_{si}$ satisfying 
the quasi-homogeneous equations $\varrho_{s'i'm}=0$. 
\end{thm}

\section{Families of symmetric semigroups}

In this section we apply the techniques developed in \cite{c2013} and \cite{St93}
(briefly described in the above section) to deal
with families of symmetric semigroups. We note that if the symmetric semigroup is generated by less
than five elements, the dimension of the moduli variety $\M$ is very known, as we noted in
section 2 of this paper. So, we must
consider symmetric semigroups of multiplicity greater than $5$, just because a symmetric
semigroup of multiplicity $m$ can be generated by $m-1$ elements. The main idea is to
adapt the techniques developed in \cite{c2013} and \cite{St93} to handle with
a projection of the (affine) canonical monomial curve over an (affine) ambient space whose dimension does not depend
on the genus $g$, but only on the multiplicity of the semigroup. Thus we are able to handle with a family
of symmetric semigroups of a given multiplicity. It is clear this approach is closely related to the
equivariant deformation theory developed by Pinkham in \cite{Pi74}.


\subsection{A family of symmetric semigroups}
 For each positive integer $\tau$ let
 \begin{eqnarray*}
    \N&:=&\langle 6, 3+6\tau, 4+6\tau, 7+6\tau, 8+6\tau\rangle\\
    &=&6\mathbb{N}\sqcup\bigsqcup_{j\in\{3, 4, 7, 8\}}^{}(j+6\tau+6\mathbb{N})\sqcup(11+12\tau+6\mathbb{N}).
    \end{eqnarray*}
\noindent be a semigroup of multiplicity $6$ generated minimally by five elements. Counting the number of gaps of 
$\N$ and picking up the largest nongap, we have
\begin{align*}
g=3+6\tau \ \mbox{and}\ l_g=12\tau+5=2g-1,
\end{align*} 
and so $\N$ is a symmetric semigroup. 

Let $\mathcal{C}$ be a complete integral Gorenstein 
curve and $P$ be a smooth point of $\C$ whose Weierstrass semigroup is $\N$ at $P$. 
For each $n\in\N$, let $x_n$ be  a rational function on $\C$ with pole divisor $nP$. 
We abbreviate
\begin{align*}
x:=x_6\ \mbox{and}\ y_j:=x_{j+6\tau}\ (j=3, 4, 7, 8)
\end{align*}
and normalize 
\begin{align*}
x_{6i}=x^i \ \mbox{and} \ x_{j+6\tau+6i}=x^iy_j, \ \ \forall\, i\geq1.
\end{align*}
A $P$-hermitian 
basis for the vector space $H^{0}(\C, (2g-2)P)$ 
consists of the functions
\begin{equation}\label{fam1base1}
\begin{array}{l}
x^0,\ldots, x^{2\tau},\\
x^0y_j,\ldots, x^{\tau}y_j\ (j=3, 4),\\
x^0y_j,\ldots, x^{\tau-1}y_j\ (j=7, 8).
\end{array}
\end{equation} 
Since the complete integral Gorenstein curve $\C$ is nonhyperelliptic, it can 
be identified with its image under the canonical embedding
\begin{align*}
(x_{n_0}: x_{n_1}:\ldots: x_{n_{g-1}}):\C\hookrightarrow\mathbb{P}^{g-1}.
\end{align*}
By considering the above the normalizations, the projection map 
\begin{align*}
  (1:x: y_3: y_4: y_7: y_8):\C\hookrightarrow\mathbb{P}^{5}
\end{align*}
defines an isomorphism of the canonical curve $\C$ onto a curve $\mathcal{D}\subset\mathbb{P}^{5}$ 
which has degree $8+6\tau$.
Instead of study the ideal of the canonical curve $\C$, which has $(g-2)(g-3)/2$ quadratic generators,
we study the ideal (and the relations between its generators) of the projected curve 
$\mathcal{D}\subset\proj^5$. The advantage is that the number of generators of the ideal of $\mathcal{D}$ 
does not depends on the genus $g$.

Let us consider a $P$-hermitian basis of the vector space $H^{0}(\C, 2(2g-2)P)$ of 
the bicanonical divisor $(4g-4)P=(24\tau+8)P$ which consists of the $3g-3$ functions
\begin{equation}\label{fam1base2}
\begin{array}{l}
x^i\ (i=0, 1,\ldots,4\tau+1 ),\\
x^iy_j\ (i=0, 1,\ldots,3\tau ,\ j=3, 4, 7, 8),\\
x^iy_3y_8\ (i=0, 1,\ldots,2\tau-1 ).
\end{array}
\end{equation}
 Let $X, Y_3, Y_4, Y_7, Y_8$ be indeterminate  whose weight we 
 attached $6, 3+6\tau, 4+6\tau, 7+6\tau, 8+6\tau$, respectively. For each $n\in\N$, 
 we define a monomial $Z_n$ of weight $n$ as follows
 \begin{align*}
 Z_{6i}=X^i,\ Z_{j+6\tau+6i}=Y_jX^i\ \mbox{and}\ Z_{11+12\tau+6i}=Y_3Y_8X^i.
 \end{align*}

 \noindent By writing the nine products $y_iy_j, (i,j)\neq(3,8)$ as linear combination of the basis 
 elements we obtain polynomials in the indeterminates $X, Y_3, Y_4, Y_7, Y_8$ that vanish 
 identically on the affine curve $\mathcal{D}\cap\mathbb{A}^5$, say 
  \begin{equation}\label{fam1pol}
   \begin{array}{ll}
  F_i=F_i^{(0)}+\sum_{j=0}^{12\tau+i}f_{ij}Z_{12\tau+i-j}& (i=6, 7, 11, 12, 14, 15)\\
   G_i=G_i^{(0)}+\sum_{j=0}^{12\tau+i}g_{ij}Z_{12\tau+i-j}& (i=8, 10, 16),
   \end{array}
   \end{equation}
 where \begin{equation*}
\begin{array}{lll}
   F_6^{(0)}= Y_3^2-X^{2\tau+1}&  F_7^{(0)}= Y_3Y_4-X^{\tau}Y_7&  G_8^{(0)}= Y_4^2-X^{\tau}Y_8,\\
    G_{10}^{(0)}= Y_3Y_7-X^{\tau+1}Y_4& F_{11}^{(0)}=Y_4Y_7-Y_3Y_8& F_{12}^{(0)}= Y_4Y_8-X^{2\tau+2},\\
     F_{14}^{(0)}= Y_7^2-X^{\tau+1}Y_8& F_{15}^{(0)}=Y_7Y_8-X^{\tau+2}Y_3& G_{16}^{(0)}=Y_8^2-X^{\tau+2}Y_4.
\end{array}
\end{equation*}
and the index $j$ only varies through integers with $12\tau+i-j\in\N$. The 
proof of the following lemma is very similar to \cite[Lemma 4.1]{c2013}.  
    
  \begin{lem}\label{lem11} The ideal of the affine 
  curve $\mathcal{D}\cap\mathbb{A}^5$ is equal to the ideal $\mathcal{I}$ generated by the 
  forms $F_i \ (i=6, 7, 11, 12, 14, 15)$ and $G_i\ (i=8, 10, 16)$. In particular, if $\C$ is
  the canonical monomial curve $\C_{\N}$, then the ideal of the affine monomial curve
  \begin{align*}
\D_{\N}\cap\mathbb{A}^5=\{(t^6, t^{3+6\tau}, t^{4+6\tau}, t^{7+6\tau}, t^{8+6\tau}):t\in\k\}
\end{align*} 
is generated by the initial forms $F_i^{(0)}\ (i=6, 7, 11, 12, 14, 15)$ and $G_i^{(0)}\ (i=8, 10, 16)$.
   \end{lem}
\begin{proof}
It is clear that $\mathcal{I}\subseteq I(\mathcal{D}\cap\mathbb{A}^5)$. Let $f$ be a polynomial 
in the variables $X, Y_3, Y_4, Y_7, Y_8$. By applying induction 
on the degree of $f$ in the indeterminates $Y_3, Y_4, Y_7, Y_8$ we note that, 
module the ideal generated by the 
nine forms $F_i, G_i$, the monomials of this polynomial $f$ are not divisible by the 
nine products $Y_iY_j, (i,j)\neq(3,8)$, hence the class of $f$ is a sum $\sum c_nZ_n$ 
of monomials $Z_n$ of pairwise different weights with $n\in\N$ and $c_n\in\k$. Thus the 
polynomial $f$ belongs to the ideal of the curve $\mathcal{D}\cap\mathbb{A}^5$ if and 
only if the linear combination $\sum c_nZ_n$ vanishes identically on the curve 
$\mathcal{D}\cap\mathbb{A}^5$ and by taking the corresponding linear combination 
$\sum c_nx_n$ of rational functions on $\k(\C)$ we have $c_n=0$ for 
each $n\in\N$, hence $f$ belongs to $\mathcal{I}$.     
\end{proof}

%
%
%


Now let us invert the above situation. Taking the fixed above symmetric semigroup $\N$, let
us consider the lifting to the polynomial ring $\k[X,Y_3,Y_4,Y_7,Y_8]$ the basis in \eqref{fam1base1}
and \eqref{fam1base2}. We now introduce nine isobaric polynomials
like in \eqref{fam1pol}. Note that the lifting of the above two basis and the nine isobaric polynomials
just depend on the semigroup $\N$. We look for relations on the coefficients $f_{ij} $ and $g_{ij}$ in
order that this nine polynomials gives rise a Gorenstein curve whose Weierstrass semigroup is $\N$ at 
the marked point. We also note that a $P$-hermitian basis of $H^0(\C,\omega)$ is not uniquely determinate by the pole divisors
$nP$ with $n\in\N$ and $n\leq 2g-2$. In this way, we keep the $P$-hermitian property by transforming
 $$\begin{array}{l}
 x\mapsto x+c_6\\
  y_3\mapsto y_3+\displaystyle\sum_{i=0}^{\tau}c_{3+6\tau}x^{\tau-i}\\
y_4\mapsto y_4+c_1y_3+\displaystyle\sum_{i=0}^{\tau}c_{4+6\tau}x^{\tau-i}\\
 y_7\mapsto y_7+c_3y_4+c_4y_3+\displaystyle\sum_{i=0}^{\tau+1}c_{1+6\tau}x^{\tau+1-i}\\
 y_8\mapsto y_8+c_1'y_7+c_4'y_4+c_5y_3+\displaystyle\sum_{i=0}^{\tau+1}c_{2+6\tau}x^{\tau+1-i},\\
 \end{array}$$ 
where $c_1, c_1', c_3, c_4, c_4', c_5$ and $c_6$ are constants of weight $1, 1, 3, 4, 4, 5$ and $6$, 
respectively. Assuming that the characteristic of $\k$ is different from $2$, and making the lifting of above 
linear change of variables, we can normalize the following coefficients
$$\begin{array}{ll}
f_{7,3+6i}=f_{11,4+6i}=0\ (i=0,\ldots,\tau)& g_{10,1+6i}=f_{11,2+6i}=0\ (i=0,\ldots,\tau+1),\\
\end{array}$$
$$\begin{array}{ll}
f_{12, 1}=f_{14, 3} =f_{15, 4} =g_{16, 5} = 0& g_{8,1}=g_{8,4}=g_{16,6}=0.\\
\end{array}$$
Due the above normalizations and those such that $c_{sir}=1$, the only freedom 
 left us is to transform $x_{n_i}\mapsto c^{n_i}x_{n_i} (i=1,\ldots, g-1)$, where 
 $c\in\k^*=\mathbb{G}_m(\k)$. By virtue of theorem \ref{teo3}, the isomorphism classes of 
 pointed Gorenstein curves $(\C,P)$ determine uniquely the coefficients up to the $\mathbb{G}_m$-action
  \begin{align*}
  g_{ij}\mapsto c^jg_{ij}\ \mbox{e}\ f_{ij}\mapsto c^jf_{ij},
  \end{align*}
 where $c\in\mathbf{k}^{*}$. We attach to coefficients $f_{ij}, g_{ij}$ the weight $j$. 
 Applying the Syzygy Lemma we get only six syzygies of the affine monomial curve $\mathcal{D}_{\N}\cap\mathbb{A}^5$
 $$\begin{array}{l}
 Y_4F_{6}^{(0)}-Y_3F_7^{(0)}+X^{\tau}G_{10}^{(0)}=0 \\
 XY_4F_{7}^{(0)}-Y_7G_{10}^{(0)}+Y_3F_{14}^{(0)}-XY_3G_{8}^{(0)}=0\\
 Y_4F_{11}^{(0)}-Y_7G_8^{(0)}+Y_8F_7^{(0)}=0\\
 Y_4F_{12}^{(0)}-Y_{8}G_8^{(0)}-X^{\tau}G_{16}^{(0)}=0\\
 Y_4F_{14}^{(0)}-Y_{8}G_{10}^{(0)}-Y_{7}F_{11}^{(0)}=0\\
 Y_4F_{15}^{(0)}-Y_8F_{11}^{(0)}-Y_{3}G_{16}^{(0)}=0.\\
 \end{array}$$
 Replacing $F_{i}^{(0)}$ and $G_{i}^{(0)}$ by $F_{i}$ and $G_{i}$
 on the above syzygies and applying the division 
 algorithm to the cubic monomials that do not belong to lifting of the basis of the vector 
 space $H^{0}(\C, 2(2g-2)P)$, the six syzygies of the affine monomial 
 curve $\mathcal{D}_{\N}\cap\mathbb{A}^5$ give rise to the following six syzygies
 module $\Lambda_3$

 \begin{flushleft}
 $\begin{array}{l}
 Y_4F_6-Y_3F_7+X^{\tau}G_{10}\equiv\\
\hspace{1cm}- \displaystyle\sum_{i=0}^{\tau-1}X^{\tau-1-i}\left( f_{6,4+6i}F_{12}+f_{6,5+6i}F_{11}-f_{7,6+6i}G_{10} \right)\\
\hspace{1cm}-\displaystyle\sum_{i=0}^{\tau}X^{\tau-i}\left(f_{6,2+6i}G_8+(f_{6,3+6i}-f_{7,3+6i})F_7-f_{7,4+6i}F_6\right),\\
 \end{array}$

 \medskip
 
$\begin{array}{l} 
 XY_4F_{7}-Y_7G_{10}+Y_3F_{14}-XY_3G_8\equiv\\
\hspace{1cm} +\displaystyle\sum_{i=0}^{\tau}X^{\tau-i}(g_{8, 1+6i}XG_{10}+g_{10, 2+6i}F_{15}+g_{10, 3+6i}F_{14})\\
\hspace{1cm}  +\displaystyle\sum_{i=0}^{\tau}X^{\tau-i}(g_{8, 5+6i}XF_6+(g_{8, 4+6i}-f_{7, 4+6i})XF_7+g_{10, 6+6i}F_{11}),\\
 \hspace{1cm} -\displaystyle\sum_{i=0}^{\tau+1}X^{\tau+1-i}\left(f_{14, 1+6i}G_{10}+f_{14, 4+6i}F_7+f_{14, 5+6i}F_6\right)\\
  \hspace{1cm} -\displaystyle\sum_{i=0}^{\tau-1}X^{\tau-1-i}\left(f_{7, 5+6i}F_{12}+f_{7, 6+6i}F_{11}\right)\\
  \end{array}$
  
  \medskip

  $\begin{array}{l}
 Y_{4}F_{11}-Y_{7}G_8+Y_8F_7\equiv\\
\hspace{1cm}  -\displaystyle\sum_{i=0}^{\tau+1}X^{\tau+1-i}(f_{11,1+6i}G_8+f_{11,2+6i}F_7)\\
 {-}\displaystyle\sum_{i=0}^{\tau}X^{\tau-i}((f_{11,3{+}6i}{+}f_{7,3{+}6i})F_{12}{-}g_{8,1{+}6i}F_{14}
{+}(f_{11,4{+}6i}{-}g_{8,4{+}6i})F_{11}{-}g_{8,5{+}6i}G_{10}), \\
\hspace{1cm}  -\displaystyle\sum_{i=0}^{\tau-1}X^{\tau-1-i}(f_{7,5+6i}G_{16}+(f_{7,6+6i}-g_{8,6+6i})F_{15}) \\ 
 \end{array}$ 
 
   \medskip

 $\begin{array}{l}
Y_4F_{12}-Y_{8}G_8-X^{\tau}G_{16}\equiv \\ 
\hspace{1cm} -\displaystyle\sum_{i=0}^{\tau}X^{\tau-i}\left(-g_{8,1+6i}F_{15}+(f_{12,4+6i}-g_{8,4+6i})F_{12}
+f_{12,5+6i}F_{11}\right) \\
\hspace{1cm} +\displaystyle\sum_{i=0}^{\tau-1}g_{8,6+6i}X^{\tau-1-i}G_{10} -\displaystyle\sum_{i=0}^{\tau+1}X^{\tau+1-i}\left(f_{12,2+6i}G_8+f_{12,3+6i}F_{7}\right), \\
 \end{array}$
 
   \medskip

 $\begin{array}{l}
 Y_4F_{14}-Y_{8}G_{10}-Y_{7}F_{11}\equiv \\
  -\displaystyle\sum_{i=0}^{\tau+1}X^{\tau+1-i}\left(f_{14,1+6i}-f_{11,1+6i})F_{11}-f_{11,2+6i}G_{10}
+f_{14,4+6i}G_{8}+f_{14,5+6i}F_{7} \right)\\
\hspace{1cm}   \displaystyle\sum_{i=0}^{\tau}X^{\tau-i}\left(g_{10,2+6i}G_{16}{+}(g_{10,6{+}6i}{-}f_{14,6{+}6i})F_{12}\right)\\
\hspace{1cm}   \displaystyle\sum_{i=0}^{\tau}X^{\tau-i}\left({+}(g_{10,3+6i}{+}f_{11,3{+}6i})F_{15}{+}f_{11,4{+}6i}F_{14}\right)\\
\end{array}$

  \medskip

$\begin{array}{l}
 Y_4F_{15}-Y_8F_{11}-Y_{3}G_{16}\equiv\\
\hspace{1cm}  {-}\displaystyle\sum_{i=0}^{\tau+1}X^{\tau+1-i}(f_{15,5+6i}G_8{-}g_{16,3+6i}G_{10}{+}f_{15,2+6i}F_{11})\\
\hspace{1cm}   -\displaystyle\sum_{i=0}^{\tau+1}X^{\tau+1-i}\left((f_{15,1+6i}-f_{11,1+6i})F_{12}+\left(f_{15,6+6i}
-g_{16,6+6i}\right)F_7\right)\\
\hspace{1cm}  +\displaystyle\sum_{i=0}^{\tau}X^{\tau-i}\left(f_{11,3+6i}G_{16}+f_{11,4+4i}F_{15} \right){+}\displaystyle\sum_{i=0}^{\tau+2}g_{16,1+6i}X^{\tau+2-i}F_6.\\
 \end{array}$
 \end{flushleft}
 
\noindent For each of the above six syzygies, the right-hand side differs from the corresponding left-hand side 
by a linear combination of basis elements of the vector space $\Lambda_3$ which are a lifting of the basis
elements of $H^{0}(\C, 3(2g-2)P)$. 
The vanishing of the coefficients of the six linear combinations 
provides quasi-homogeneous equations between the coefficients $f_{ij}$ and $g_{ij}$. To 
express these equations in a concise manner we introduce polynomials in only one variable. 
For each $i=6, 7, 11, 12, 14, 15$ let us consider the polynomial
 \begin{align*}
 f_i:=\displaystyle\sum_{r=1}^{12\tau+i}F_i(t^{-6}, t^{-6-3\tau}, t^{-6-4\tau}, 
 t^{-6-7\tau}, t^{-8-3\tau})t^{i+12\tau}\
 \end{align*}
 and we write each one as the sum of its partial polynomials 
 \begin{align*}
 f_i^{(j)}=\displaystyle\sum_{r\equiv j\mod{6}}^{}f_{ir}t^r,\ (j=1,\ldots,6),
 \end{align*}
 which are defined by collecting every terms whose exponents are in the 
 same residue class module $6$. Analogously we define the polynomials $g_j\ (j=8, 10, 16)$ 
 and its partial polynomials $g_i^{(j)}$. 
 
 From our normalizations of the constants 
 $f_{ij}, g_{ij}$, the eight partial polynomials $f_{12}^{(1)}, f_{14}^{(3)}, 
 f_{15}^{(4)}, g_{16}^{(5)}, f_{7}^{(3)}, f_{11}^{(4)}, g_{10}^{(1)}$ and 
 $f_{12}^{(2)}$ are equal to zero, and so we may express each $g_j$ and $f_j$ in terms 
 of the remaining $41$ partial polynomials. 
   
 
The formal degree of the partial polynomials with $i=j$ and $i-j\equiv6$ 
(namely: $f_6^{(6)}, f_7^{(1)}, g_8^{(2)}, g_{10}^{(4)}, f_{11}^{(5)}, f_{12}^{(6)}$ 
and $f_{12}^{(4)}, f_{15}^{(3)}, g_{16}^{(4)}$) is $i+12\tau$. The partial 
polynomials $f_6^{(4)}, f_6^{(5)}, f_7^{(5)}, f_7^{(6)}, g_8^{(1)}$ and $g_{16}^{(1)}$ 
have formal degree $j+6(\tau-1),$ and $13+6\tau$, respectively. Among the remaining $26$ 
polynomials, $13$ have formal degree $j+6\tau$ and the other ones 
have formal degree $j+6(\tau+1)$. Therefore, the number of the coefficients that 
are still involved is equal to
 \begin{align*}
 (2\tau+1)+5(2\tau+2)+3(2\tau+3)+6\tau+3+13(\tau+1)+13(\tau+2)-3=50\tau+59,
 \end{align*} 
 where the subtraction by three corresponds to the normalizations $g_{8,1}=g_{8,4}=g_{16,6}=0$. 
By virtue of theorem \ref{teo3} we get an explicit construction of a compactification 
 of moduli space $\M$, with $\N:=<6,3+6\tau,4+6\tau, 7+6\tau, 8+6\tau>$ for each $\tau\geq 1$,
 as follows.
 \begin{thm}\label{thm1}
 Let $\N$ be the semigroup generated by $6, 3+6\tau, 4+6\tau, 7+6\tau$ and $8+6\tau$ 
 where $\tau$ is a positive integer. The isomorphism classes of the pointed complete 
 integral Gorenstein curves with Weierstrass semigroup $\N$ correspond bijectively to 
 the orbits of the $\mathbb{G}_m$-action on the quasi-cone of the vectors of length $50\tau+59$ 
 whose coordinates are the coefficients $g_{ij}, f_{ij}$ of the $41$ partial polynomials 
 that satisfy the six equations:
 
 $\begin{array}{lcl}
 f_6-f_7+g_{10}&=&-\f[6,2]g_8-(\f[6,3]-\f[7,3])f_7+\f[7,4]f_6-\f[6,4]f_{12}\\ & &-\f[6,5]f_{11}+\f[7,6]g_{10},\\
  \\
 f_7-g_{10}+f_{14}-g_8&=&(\g[8,1]-\f[14,1])g_{10}+\g[10,2]f_{15}+\g[10,3]f_{14}+(\g[10,6]-\f[7,6])f_{11}\\
 & &+(\g[8,4]-\f[7,4]-\f[14,4])f_7+ (\g[8,5]-\f[14,5])f_6-\f[7,5]f_{12},\\
 
 \\
f_{11}-g_8+f_7&=& \g[8,5]g_{10}-\f[11,1]g_8-\f[11,2]f_7-\f[7,5]g_{16}+(\g[8,6]-\f[7,6])f_{15}\\
 & &+\g[8,1]f_{14}-(\f[11,3]+\f[7,3])f_{12}-(\f[11,4]-\g[8,4])f_{11}, \\
 
 \\
 
 f_{12}-g_8-g_{16}&=& \g[8,1]f_{15}-(\f[12,4]-\g[8,4])f_{12}-\f[12,5]f_{11}+\g[8,6]g_{10}\\ & & -\f[12,2]g_{8}-\f[12,3]f_7, \\
 
 \\
 
 f_{14}-g_{10}-f_{11}&=& (\f[11,1]-\f[14,1])f_{11}+\f[11,2]g_{10}-\f[14,4]g_{8}-\f[14,5]f_7+\g[10,2]g_{16}\\
 & &+(\g[10,3]+\f[11,3])f_{15}+\f[11,4]f_{14}-(\f[14,6]-\g[10,6])f_{12}, \\
 
 \\
 
 f_{15}-f_{11}-g_{16}&=&(\f[11,1]-\f[15,1])f_{12}-(\f[15,6]-\g[16,6])f_{7}+\g[16,3]g_{10}-\f[15,5]g_{8}\\
 & &-\f[15,2]f_{11}+\g[16,1]f_{6}+\f[11,3]g_{16}+\f[11,4]f_{15}.
 \end{array}$
 \end{thm}
 Note that the compactified moduli space $\overline{\mathscr{M}_{g,1}^\N}$ can 
 be embedded into a weighted projective space of dimension $50\tau+58$.
Now the key is diminish the dimension of the ambient 
space by projecting this space onto a space of lower dimension. Initially, we take the 
six equations of the moduli space given by the above theorem \ref{thm1} and rewritten this 
equations in terms of $36$ polynomial equations between $41$ partial polynomials. 
Among this equations, there are the following six linear equations between the partial polynomials
\begin{align*}
f_7^{(5)}=f_6^{(5)}, f_{14}^{(5)}=g_8^{(5)}-f_6^{(5)}, g_8^{(4)}=f_7^{(4)}, f_{12}^{(5)}=g_8^{(5)}, f_{14}^{(1)}=f_{11}^{(1)}, g_{16}^{(1)}=f_{15}^{(1)}-f_{11}^{(1)}.
\end{align*}
With this normalizations we diminish the dimension of the ambient space to $44\tau+50$. 
By analyzing the formal degree in the remaining $30$ equations we can eliminate more partial polynomials, 
until the remaining quasi-homogeneous equations do not admit linear terms. However, even with 
this procedure, the solution of the remaining polynomial equations are far from being practicable for every $\tau\geq 1$, 
even with a computer.


We avoid this technical difficulty by considering a much simpler algebraic
space which contains the moduli variety $\M$, which consists of a space
given by only the forms of degree $2$ of the generators of the ideal of the moduli variety $\M$.
First we determine the vector space 
$T_{\k[\N]_{}\k}^{1,-}$ which is, up to an isomorphism, the locus of the 
linearizations of the $36$ equations between the partial polynomials. Indeed the linearizations
consist in substituting the right hand side of the equations in theorem \ref{teo3} by
zeros and solving the linear systems in terms of the partial polynomials.
We can solving this system as follows.
 $$\begin{array}{c}
 \f[7, 1] =\f[15, 1] = 0,\ \f[11, 1] = \g[8, 1],\ \f[14, 1] = \g[8, 1],\ \g[16, 1] = -\g[8, 1]; \\
 \g[8, 2] = 0,\ \g[10, 2] = \f[6, 2],  \ \f[14, 2] = \f[6, 2],\ \f[15, 2] = \f[12, 2], \ 
 \g[16, 2] = \f[12, 2]; \\
 \f[6, 3] =\f[11, 3] =\g[10, 3]=0, \ \f[15, 3] = \f[12, 3], \g[16, 3] = \f[12, 3]; \\
 \g[16, 4] = 0, \ \f[7, 4] =\g[8, 4], \ \g[10, 4] = \f[6, 4]-\g[8, 4], \ \f[12, 4] = \g[8, 4], \ 
 \f[14, 4] = \f[6, 4]-\g[8, 4]; \\
 \f[7, 5] = \f[6, 5], \ \f[14, 5]=\f[11, 5] = -\f[6, 5]+\g[8, 5], \ \f[12, 5] = \g[8, 5], 
 \ \f[15, 5] = -\f[6, 5]+\g[8, 5]; \\
 \f[6, 6] = \f[7, 6]+\g[10, 6],\ \g[8, 6] = \f[7, 6], \ \f[12, 6] = \f[7, 6]+\g[16, 6], 
 \ \f[14, 6] = \g[10, 6], \ \f[15, 6] = \g[16, 6].
 \end{array}$$
 
\noindent Thus we conclude that the vector space $T_{\k[\N]_{}\k}^{1,-}$ can be identified 
 with the space whose entries are the coefficients of the remaining partial polynomials
 \begin{align*}
\g[8,1], \ \f[6,2], \ \f[12,2], \ \f[12,3], \ \g[8,4],\ \f[6,4], \ \f[6,5], \ \g[8,5], 
\ \f[7,6], \ \g[10,6], \ \g[16,6].
 \end{align*}
By counting the coefficients of this partial polynomials we have $11\tau+11$ 
coefficients, and discounting the conditions corresponding to the three normalizations
\begin{align*}
g_{8,1}=g_{8,4}=g_{16,6}=0,
\end{align*} 
we obtain
\begin{align*}
\dim T_{\k[\N]_{}\k}^{1,-}=11\tau+8.
\end{align*}
 
  

\noindent Now, if we enter with the linearizations into the right hand side of the equations
in theorem \ref{teo3}, then we can solve this equations in terms of the partial polynomials
which appear only in the linearizations. We solve in a way that on the right hand-side
only appear the partial polynomials in the linearizations and the corresponding left
hand-side only the partial polynomials which are not in the linearizations. 
Collecting those equations whose formal degrees of the left hand-sides are less than the corresponding formal
degrees of the right hand-sides, we obtain
\begin{equation*}
\begin{array}{l}
\f[14, 1] =-\f[6, 4]\f[12, 3]-\f[6, 5]\f[12, 2]+(\g[10,6]-\f[7,6])\g[8, 1]+\g[8, 1]\\
\g[16, 1] =\g[8,1](\f[7, 6]-\g[16,6])+\f[12, 2]\g[8, 5]+\f[12, 3]\g[8, 4]-\g[8, 1]\\
\f[11, 3] =\f[6, 2]\g[8, 1]+\f[6, 4]\g[8, 5]-\f[6, 5]\g[8, 4]\\
\f[14, 4] =\f[6, 2]\f[12, 2]-\f[6, 4](\f[7, 6]-\g[16, 6])-\g[8, 4](\g[10, 6]-\f[7,6])+\f[6, 4]-\g[8, 4]\\
\f[14,5]= \f[6, 2]\f[12, 3]+\f[6, 5](\f[7, 6]-\g[16, 6])+\g[8, 5](\g[10, 6]-\f[7,6])-\f[6, 5]+\g[8, 5].
\end{array}
\end{equation*}

\noindent By comparing the formal degrees of the left with the right hand-sides of the above five equations,
we introduce the following polynomials equations in the coefficients $f_{ij}$ and $g_{ij}$.
\begin{equation}\label{fam1quasi}
\begin{array}{ll}
\pi_{7+6\tau}(-\f[6, 4]\f[12, 3]-\f[6, 5]\f[12, 2]+\tilde{g}_{10}^{(6)}\g[8, 1])=0\\
\pi_{13+6\tau}(-\g[8,1]\tilde{g}_{16}^{(6)}+\f[12, 2]\g[8, 5]+\f[12, 3]\g[8, 4])=0\\
\pi_{3+6\tau}(\f[6, 2]\g[8, 1]+\f[6, 4]\g[8, 5]-\f[6, 5]\g[8, 4])=0\\
\pi_{10+6\tau}(\f[6, 2]\f[12, 2]+\f[6, 4]\tilde{g}_{16}^{(6)}-\g[8, 4]\tilde{g}_{10}^{(6)})=0\\
\pi_{11+6\tau}( \f[6, 2]\f[12, 3]-\f[6, 5]\tilde{g}_{16}^{(6)}+\g[8, 5]\tilde{g}_{10}^{(6)})=0,
\end{array}
\end{equation}
where $\tilde{g}_{10}^{(6)}=\g[10,6]-\f[7,6], \tilde{g}_{16}^{(6)}=\g[16,6]-\f[7, 6]$
and $\pi_i$ denotes the projection operator in $t$ that annihilates the terms of degree not greater 
than $i$. The above polynomials equations give rise to an affine quadratic quasi-cone 
$\mathcal{Q}\subset\mathbb{A}^{11\tau+8}$, 
which contains the affine quasi-cone $\mathcal{X}\subset\mathbb{A}^{11\tau+8}$, where this last one is such that 
$\M\cong \mathcal{X}/\mathbb{G}_{m}$.
Hence $\dim\mathcal{Q}\geq\dim\mathcal{X}=\dim\M+1$. This is the method
presented in \cite[\& 3]{c2013}.

We note that the congruences in \eqref{fam1quasi} does not depend on the coefficients 
$$f_{6,2},  f_{12,2}, f_{12,3}, g_{8,5},\tilde{g}_{10,6}, \tilde{g}_{16,6}, f_{12,8}, f_{12,9},
\tilde{g}_{16,12} 
\ \mbox{and}\  f_{7,6i}, i=1,\ldots, \tau-1.$$
These congruences depend only on $10\tau$ coefficients. They can be expressed in five equations 
between ten elements of the $\tau$-dimensional artinian algebra
\begin{align*}
A:=k[\epsilon]=\displaystyle\bigoplus_{j=0}^{\tau-1}k\epsilon^{j},\ \mbox{where}\ \epsilon^{\tau}=0.
\end{align*} 
\begin{thm}\label{teo21}
The quadratic quasi-cone $\mathcal{Q}$ is isomorphic to the direct product
\begin{align*}
\mathcal{Q}=M\times N,
\end{align*}
where $M$ is the $(\tau+8)$-dimensional weighted space of weights $2, 2, 3, 5, 6, 6, 8, 9, 12$ 
and $6i, i=1,\ldots,\tau-1$, and $N$ is the quadratic quasi-cone consisting of vectors
\begin{align*}
(\omega_1,\ldots,\omega_{10})=\left(\displaystyle\sum_{j=0}^{\tau-1}\omega_{1j}\epsilon^j,\ldots,
\displaystyle\sum_{j=0}^{\tau-1}\omega_{10,j}\epsilon^j\right),
\end{align*}
such that satisfying the five equations
\begin{equation*}
\begin{array}{ll}
\omega_4\omega_9-\omega_3\omega_7-\omega_2\omega_8=0,\\
\omega_6\omega_7-\omega_4\omega_{10}+\omega_5\omega_8=0,\\
\omega_1\omega_{4}+\omega_2\omega_6-\omega_3\omega_5=0,\\
\omega_1\omega_{7}+\omega_2\omega_{10}-\omega_5\omega_9=0,\\
\omega_6\omega_9-\omega_3\omega_{10}+\omega_1\omega_8=0,\\
\end{array}
\end{equation*}
in the artinian algebra $A$. 
\end{thm}
\begin{proof}
Defining
$$\begin{array}{ll}
\omega_{1j}=f_{6,6\tau+2-6j}& \omega_{2j}=f_{6,6\tau-2-6j},\\
\omega_{3j}=f_{6,6\tau-1-6j}&\omega_{4j}=g_{8,6\tau-5-6j},\\
\omega_{5j}=g_{8,6\tau-2-6j}&\omega_{6j}=g_{8,6\tau+5-6j},\\
\omega_{7j}=f_{12,6\tau+8-6j}&\omega_{8j}=f_{12,6\tau+9-6j},\\
\omega_{9j}=\tilde{g}_{10,6\tau+6-6j}& \omega_{10,j}=\tilde{g}_{16,6\tau+6-6j},\\
\end{array}$$
it is sufficient to observe that the conditions on the $10\tau$ coefficients are 
equivalents to the five quadratic equations in the artinian algebra $A$.
\end{proof}
By applying induction on $\tau$ one can proof that
\begin{cor}\label{cor13}
$\dim\mathcal{Q}=8\tau+8$ and hence
\begin{align*}
\dim\M\leq 8\tau+7.
\end{align*}
\end{cor}

 \subsection{A second family of symmetric semigroups}

 We apply the same method above for the following particular family
 of symmetric semigroups. For each $\tau\geq 1$, let
 
 \begin{eqnarray*}
\N&=&\langle 6, 7+6\tau, 8+6\tau, 9+6\tau, 10+6\tau\rangle\\
&=&\mathbb{N}\sqcup\bigsqcup_{j\in\{7, 8, 9, 10\}}^{}(j+6\tau+6\mathbb{N})
\sqcup(17+12\tau+6\mathbb{N}),
\end{eqnarray*}
be a symmetric semigroup of genus $g=6+6\tau$. We note if $\tau$ is
equal to zero, $\N$ is also a symmetric semigroup of multiplicity $6$ generated
minimally by five elements. However the ideal of the canonical monomial
curve $\C_{\N}$ can not be generate by only quadratic forms, this special cases
were treated in a recent preprint by Contiero and Fontes in \cite{CF}.
Since $\tau\geq 1$ and the method is the same of the preceding subsection,
we make a lot of shortcuts and we do not explain the method again.

Let $\C$ 
be a complete integral 
Gorenstein curve and $P$ be a nonsingular point of $\C$ whose Weierstrass 
semigroup at $P$ is $\N$. For each $n\in\N$ let $x_n$ be a rational function 
on $\C$ with pole divisor $nP$. Let us consider
\begin{align*}
x:=x_6\ \mbox{and}\ y_j:=x_{j+6\tau}\ (j=7, 8, 9, 10) \ \mbox{with} \ x_{6i}=x^i, x_{j+6\tau+6i}=x^iy_j,\ \  i\geq1. 
\end{align*}
Since $\C$ is a nonhyperelliptic curve, it can be canonically embedded in $\mathbb{P}^{g-1}$, 
and the projection map 
\begin{align*}
(1:x: y_7: y_8: y_9: y_{10}):\C\hookrightarrow\mathbb{P}^{5}
\end{align*}
defines an isomorphism of the canonical curve $\C$ onto a curve 
$\mathcal{D}\subset\mathbb{P}^5$ of degree $6\tau+10$. 
A $P$-hermitian basis of the vector space $H^0(\C,(4g-4)P)$ is
\begin{eqnarray*}
&& x^i\ (i=0, 1,\ldots,4\tau+3 ),\\
&& x^iy_j\ (i=0, 1,\ldots,3\tau+2 ,\ j=7, 8),\\
&& x^iy_j\ (i=0, 1,\ldots,3\tau+1 ,\ j=9, 10),\\
&& x^iy_7y_{10}\ (i=0, 1,\ldots,2\tau ).
\end{eqnarray*}

\noindent For a nongap $n\in\N$, the monomial 
$Z_n$ of weight $n$ is
\begin{align*}
      Z_{6i}=X^i,\ Z_{j+6\tau+6i}=Y_jX^i\ \mbox{and}\ Z_{11+12\tau+6i}=Y_7Y_{10}X^i.
      \end{align*}
    
\noindent Writing the nine products $y_iy_j, (i, j)\neq(7, 10)$ 
as linear combination of the basis elements of the $\k$-vector space $H^0(\C, 2(2g-2)P)$ we obtain,
in the indeterminates $X, Y_7, Y_8, Y_9, Y_{10}$, the polynomials
      \begin{eqnarray*}
       F_i=F_i^{(0)}+\sum_{j=0}^{12\tau+i}f_{ij}Z_{12\tau+i-j}& (i=14, 15, 16, 17, 18)\\
        G_i=G_i^{(0)}+\sum_{j=0}^{12\tau+i}g_{ij}Z_{12\tau+i-j}& (i=16, 18, 19, 20),
        \end{eqnarray*}    
that vanish identically on the affine curve $\mathcal{D}\cap\mathbb{A}^5$, where 
\begin{equation*}
 \begin{array}{lll}
F_{14}^{(0)}= Y_7^2-X^{\tau+1}Y_8&  F_{15}^{(0)}= Y_7Y_8-X^{\tau+1}Y_9&  F_{16}^{(0)}=
Y_7Y_9-X^{\tau+1}Y_{10},\\
G_{16}^{(0)}= Y_8^2-X^{\tau+1}Y_{10}& F_{17}^{(0)}=Y_8Y_9-Y_7Y_{10}& F_{18}^{(0)}= 
Y_8Y_{10}-X^{2\tau+3},\\
G_{18}^{(0)}= Y_9^2-X^{2\tau+3}& G_{19}^{(0)}=Y_9Y_{10}-X^{\tau+2}Y_7& G_{20}^{(0)}=
Y_{10}^2-X^{\tau+2}Y_8.
\end{array}
\end{equation*}

  \begin{lem}\label{lem21} The ideal of the affine 
  curve $\mathcal{D}\cap\mathbb{A}^5$ is equal to the ideal $\mathcal{I}$ generated by the above
  forms $F_i$ and $G_i$. In particular, if $\C$ is
  the canonical monomial curve $\C_{\N}$, then the ideal of the affine monomial curve
  \begin{align*}
\D_{\N}\cap\mathbb{A}^5=\{(t^6, t^{7+6\tau}, t^{8+6\tau}, t^{9+6\tau}, t^{10+6\tau}):t\in\k\}
\end{align*} 
is generated by the initial forms $F_i^{(0)}$ and $G_i^{(0)}$.
   \end{lem}


Inverting the above situation and considering the polynomials $F_i$ and $G_{i}$ just induced by
the semigroup $\N$, we normalize the coefficients
$$\begin{array}{ll}
f_{18,1}=g_{18,1}=g_{19,2}=g_{20,3}=0,& f_{15,6}=f_{16,2}=g_{16,1}=0\\
\end{array}$$
and 
$$\begin{array}{ll}
    f_{16,1+6i}=f_{17,4+6i}=f_{18,2+6i}=g_{19,3+6i}=0,&\quad (i=0,...,\tau +1).\\
\end{array}$$
By applying the Syzygy Lemma we get 
$$ \begin{array}{l}
Y_{10}F_{14}^{(0)}-Y_8F_{16}^{(0)}+Y_7F_{17}^{(0)}=0,\\
Y_{10}F_{15}^{(0)}-Y_9G_{16}^{(0)}+Y_8F_{17}^{(0)}=0,\\
Y_{10}G_{16}^{(0)}-Y_{8}F_{18}^{(0)}-X^{\tau +1}G_{20}=0,\\
Y_{10}F_{17}^{(0)}-Y_8G_{19}^{(0)}+Y_7G_{20}^{(0)}=0,\\
Y_{10}F_{18}^{(0)}-X^{\tau +2}G_{16}^{(0)}-Y_8G_{20}^{(0)}=0,\\
Y_{10}G_{18}^{(0)}-X^{\tau +2}F_{16}^{(0)}-Y_9F_{19}^{(0)}=0,\\
Y_{10}G_{19}^{(0)}-X^{\tau +2}F_{17}^{(0)}-Y_9G_{20}^{(0)}=0.\\
      \end{array}$$
      
\noindent Hence we obtain seven polynomial equations module $\Lambda_3$
      \begin{flushleft}
      $\begin{array}{l}
      Y_{10}F_{14}-Y_8F_{16}+Y_7F_{17}\equiv\\
 \hspace{1cm}     -\displaystyle\sum_{i=0}^{\tau +1}X^{\tau +1-i}[(f_{17, 3+6i}-f_{16,3+6i})F_{15}+
      f_{17,2+6i}F_{16}-f_{16, 2+6i}G_{16}]\\
 \hspace{1cm}     -\displaystyle\sum_{i=0}^{\tau }X^{\tau -i}[(f_{14,6+6i}-f_{16,6+6i})F_{18}
      +f_{14, 5+6i}G_{19}+f_{14, 4+6i}G_{20}],
      \end{array}$
      
\medskip      
      
      $\begin{array}{l}
      Y_{10}F_{15}-Y_9G_{16}+Y_8F_{17}\equiv\\
\hspace{1cm}      \displaystyle\sum_{i=0}^{\tau }X^{\tau -i}
      [(g_{16,6+6i}-f_{15,6+6i})G_{19}-f_{15, 5+6i}G_{20}]\\
\hspace{1cm}      +\displaystyle\sum_{i=0}^{\tau +1}X^{\tau +1-i}[g_{16,3+6i}F_{16}-f_{17, 3+6i}G_{16}-
      (f_{17,2+6i}-g_{16, 2+6i})F_{17}\\
\hspace{1cm}      -(f_{15,1+6i}+f_{17,1+6i})F_{18}+g_{16,1+6i}G_{18}],
      \end{array}$
      
\medskip

      $\begin{array}{l}
      Y_{10}G_{16}-Y_8F_{18}+X^{\tau +1}G_{20}\equiv
      -\displaystyle\sum_{i=0}^{\tau }X^{\tau -i}
      g_{16,6+6i}G_{20}+\\ 
     \displaystyle\sum_{i=0}^{\tau +1}X^{\tau +1-i}[f_{18, 5+6i}F_{15}+f_{18, 4+6i}G_{16}+
      f_{18,3+6i}F_{17}
      -g_{16,2+6i}F_{18}-g_{16,1+6i}G_{19}], 
      \end{array}$
      
          
\medskip          
      
      $\begin{array}{l}
      Y_{10}F_{17}-Y_8G_{19}+Y_7G_{20}\equiv\\
     +\displaystyle\sum_{i=0}^{\tau +1}X^{\tau +1-i}[(g_{19, 6+6i}-g_{20,6+6i})F_{15}-g_{20,5+6i}F_{16}
      +g_{19,5+6i}G_{16}+g_{19, 4+6i}F_{17}\\
\hspace{1cm}      -f_{17,3+6i}F_{18}-f_{17,2+6i}G_{19}-f_{17,1+6i}G_{20}]-
      \displaystyle\sum_{i=0}^{\tau +2}X^{\tau +2-i}g_{20, 1+6i}F_{14},
      \end{array}$

      $\begin{array}{l}
      Y_{10}F_{18}-X^{\tau+2}G_{16}-Y_8G_{20}\equiv
      \displaystyle\sum_{i=0}^{\tau +2}X^{\tau +2-i}g_{20,1+6i}F_{15}\\
    {+}\displaystyle\sum_{i=0}^{\tau +1}X^{\tau +1-i}[g_{20,6+6i}G_{16}{+}g_{20, 5+6i}F_{17}{-}(f_{18,4+6i}{-}g_{20,4+6i})F_{18}{-}f_{18,3+6i}G_{19}],\\
      \end{array}$
      
\medskip           
      
      $\begin{array}{l}
      Y_{10}G_{18}-X^{\tau+2}F_{16}-Y_9G_{19}\equiv
      \displaystyle\sum_{i=0}^{\tau +1}X^{\tau +1-i}[g_{19, 6+6i}F_{16}+g_{19,5+6i}F_{17}\\
\hspace{1cm}      +g_{19, 4+6i}G_{18}-g_{18,4+6i}F_{18}-g_{18,3+6i}G_{19}-g_{18, 2+6i}G_{20}], 
      \end{array}$
      
\medskip      
      
      $\begin{array}{l}
      Y_{10}G_{19}-X^{\tau+2}F_{17}-Y_9G_{20}\equiv
      \displaystyle\sum_{i=0}^{\tau +2}X^{\tau +2-i}g_{20,1+6i}F_{16}\\
      {+}\displaystyle\sum_{i=0}^{\tau +1}X^{\tau +1-i}[g_{20,6+6i}F_{17}{+}g_{20, 5+6i}G_{18}{-}g_{19,5+6i}F_{18}{-}(g_{19,4+6i}{-}g_{20,4+6i})G_{19}].\\
      
      \end{array}$
      \end{flushleft}
\begin{thm}\label{thm11}
Let $\N$ be the semigroup generated by $6, 7+6\tau, 8+6\tau, 9+6\tau$ and $10+6\tau$ where
$\tau$ is a positive integer. The isomorphism classes of the pointed complete integral 
Gorenstein curves with Weierstrass semigroup $\N$ correspond bijectively  to the orbits 
of the $\mathbb{G}_m$-action on the quasi-cone of the vectors of length $50\tau+84$ whose 
coordinates are the coefficients $g_{ij}, f_{ij}$ of the $41$ partial polynomials that 
satisfy the seven equations:
\begin{flushleft}

$\begin{array}{l}     
f_{18}-g_{16}-g_{20}=g_{20}^{(6)}g_{16}+g_{20}^{(5)}f_{17}-(f_{18}^{(4)}-g_{20}^{(4)})f_{18}
-f_{18}^{(3)}g_{19}-f_{18}^{(2)}g_{20}+g_{20}^{(1)}f_{15},\\
 \end{array}$
 
\medskip
 
 $\begin{array}{l}     
g_{18}-f_{16}-g_{19}=g_{19}^{(6)}f_{16}+g_{19}^{(5)}f_{17}+g_{19}^{(4)}g_{18}
-g_{18}^{(4)}f_{18}-g_{18}^{(3)}g_{19}-g_{18}^{(2)}g_{20},\\
 \end{array}$

\medskip 
 
$\begin{array}{l}
g_{19}-f_{17}-g_{20}=g_{20}^{(6)}f_{17}+g_{20}^{(5)}g_{18}-g_{19}^{(5)}f_{18}
-(g_{19}^{(4)}-g_{20}^{(4)})g_{19}+g_{20}^{(1)}f_{16},
\end{array}$

\medskip

$\begin{array}{ll}   
g_{16}-f_{18}+g_{20}=& f_{18}^{(5)}f_{15}+f_{18}^{(4)}g_{16}+f_{18}^{(3)}f_{17}
-g_{16}^{(2)}f_{18}-g_{16}^{(1)}g_{19}-g_{16}^{(6)}g_{20},\\
\end{array}$ 

\medskip

$\begin{array}{ll}
f_{14}-f_{16}+f_{17}=& f_{16}^{(2)}g_{16}-(f_{17}^{(3)}-f_{16}^{(3)})f_{15}-f_{17}^{(2)}f_{16}-
(f_{14}^{(6)}-f_{16}^{(6)})f_{18}\\
&-f_{14}^{(5)}g_{19}-f_{14}^{(4)}g_{20},
 \end{array}$
 
\medskip 
 
 $\begin{array}{ll}   
f_{15}-g_{16}+f_{17}= &g_{16}^{(1)}g_{18}-(f_{15}^{(1)}+f_{17}^{(1)})f_{18}+g_{16}^{(3)}f_{16}
-f_{17}^{(3)}g_{16}-f_{15}^{(5)}g_{20}\\
&-(f_{17}^{(2)}-g_{16}^{(2)})f_{17}-(f_{15}^{(6)}-g_{16}^{(6)})g_{19},\\
 \end{array}$
 
\medskip
 
$\begin{array}{ll}
f_{17}-g_{19}+g_{20}=&g_{19}^{(4)}f_{17}+g_{19}^{(5)}g_{16}-g_{20}^{(5)}f_{16}
-(g_{20}^{(6)}-g_{19}^{(6)})f_{15}\\
&-f_{17}^{(3)}f_{18}-f_{17}^{(2)}g_{19}-f_{17}^{(1)}g_{20}-g_{20}^{(1)}f_{14}.\\
 \end{array}$

\end{flushleft}
      \end{thm}

The linearizations depend only on the $11$ following partial polynomials

$$f_{14}^{(1)}, f_{17}^{(2)}, 
f_{16}^{(2)}, g_{16}^{(3)}, f_{14}^{(4)}, g_{20}^{(4)}, f_{14}^{(5)}, f_{18}^{(5)}, 
f_{14}^{(6)}, f_{15}^{(6)} \mbox{ and }g_{20}^{(6)}.$$

\noindent Counting its coefficients and discounting the three 
normalizations $f_{15,6}=$ $f_{16,2}=$ $g_{16,1}=0$, we obtain $11\tau+15$ coefficients. Thus
    \begin{align*} 
    \dim T_{\k[\N]_{}\k}^{1,-}=11\tau+15.
    \end{align*}


\noindent The equations of the affine quadratic quasicone $\mathcal{Q}$ are given by:

     \begin{equation}\label{thmfinal}
\begin{array}{l}
      \pi_{13+6\tau}(f_{14}^{(1)}\tilde{g}_{20}^{(6)}+f_{14}^{(4)}g_{16}^{(3)}+
      f_{14}^{(5)}f_{17}^{(2)}-f_{17}^{(2)}f_{18}^{(5)}+g_{16}^{(3)}g_{20}^{(4)})=0\\
     
     \pi_{7+6\tau}(-f_{14}^{(1)}\tilde{f}_{15}^{(6)}-f_{14}^{(4)}g_{16}^{(3)}
     -f_{14}^{(5)}f_{17}^{(2)})=0\\
     
     \pi_{9+6\tau}(-f_{14}^{(1)}f_{16}^{(2)}-f_{14}^{(4)}f_{18}^{(5)}-f_{14}^{(5)}g_{20}^{(4)})=0\\
     
     \pi_{10+6\tau}(f_{16}^{(2)}f_{17}^{(2)}+f_{14}^{(4)}(\tilde{g}_{20}^{(6)}-\tilde{f}_{15}^{(6)})-g_{20}^{(4)}
     \tilde{f}_{15}^{(6)})=0\\
     
     \pi_{11+6\tau}(f_{14}^{(5)}\tilde{f}_{15}^{(6)}-f_{14}^{(5)}\tilde{g}_{20}^{(6)}
     -\tilde{f}_{15}^{(6)}f_{18}^{(5)}+f_{16}^{(2)}g_{16}^{(3)})=0,\\
     \end{array}
\end{equation}
where $\tilde{g}_{20}^{(6)}=g_{20}^{(6)}-f_{14}^{(6)}$, $\tilde{f}_{15}^{(6)}=
f_{15}^{(6)}-f_{14}^{(6)}$ and $\pi_i$ denotes the projection operator in $t$ 
that annihilates the terms of degree not greater than $i$. We can observe that these 
equations does not depend of the coefficients $f_{14,1}, f_{17,2}, g_{16,3}, 
g_{20,4}, f_{18,5}$ and $f_{14,6i}, i=2,\ldots,\tau+1$. By considering the 
$(\tau+1)$-dimensional artinian algebra
\begin{align*}
A:=k[\epsilon]=\displaystyle\bigoplus_{j=0}^{\tau}k\epsilon^{j},\ \mbox{where}\ \epsilon^{\tau+1}=0,
\end{align*} 
we can write the equations in \eqref{thmfinal} in terms of five polynomial equations 
between $\tau+1$ elements of the $A$.
\begin{thm}\label{teo22}
The quadratic quasi-cone $\mathcal{Q}$ is isomorphic to the direct product
\begin{align*}
\mathcal{Q}=M\times N,
\end{align*}
where $M$ is the $(\tau+5)$-dimensional weighted space of weights $1, 2, 3, 4, 5$ 
and $6i, i=2,\ldots,\tau+1$, and $N$ is the quadratic quasi-cone consisting of vectors
\begin{align*}
(\omega_1,\ldots,\omega_{10})=\left(\displaystyle\sum_{j=0}^{\tau}\omega_{1j}\epsilon^j,\ldots,
\displaystyle\sum_{j=0}^{\tau}\omega_{10,j}\epsilon^j\right),
\end{align*}
such that satisfying the five equations
\begin{equation*}
\begin{array}{ll}
\omega_1\omega_{10}+\omega_4\omega_5+\omega_3\omega_7-\omega_3\omega_8+\omega_4\omega_6=0,\\
-\omega_1\omega_9-\omega_4\omega_{5}-\omega_3\omega_7=0,\\
-\omega_1\omega_{2}-\omega_5\omega_8-\omega_6\omega_7=0,\\
\omega_2\omega_{3}+\omega_5(\omega_{10}-\omega_9)-\omega_6\omega_9=0,\\
\omega_7\omega_9-\omega_7\omega_{10}-\omega_8\omega_9+\omega_2\omega_4=0,\\
\end{array}
\end{equation*}
in the artinian algebra $A$. 
\end{thm}    



\begin{cor}\label{boundfam2}
We have $\dim\mathfrak{Q}_\N=8\tau+12$. Thus
\begin{align*}
\dim\M\leq 8\tau+11.
\end{align*}
\end{cor}
     
     
 \section{Collecting known dimensions}

As noted in section 2 of this paper, if a numerical semigroup $\N$ is negatively graded then 
the dimension of $\M$ is equal to Deligne--Pinkham's upper bound $2g-2+\lambda(\N)$, which
is also equal to Pflueger's bound $3g-2-\mathrm{ewt}(\N)$.
Additionally, it is also know, cf. \cite{N16}, that the dimension of $\M$ is equal to Pflueger's bound for
all numerical semigroups whose genus is not greater than $6$.

In the following table \ref{tab1} we collect all numerical semigroups of genus $g\leq 6$
and compare the bounds given by Deligne--Pinkham and Pflueger. Of course we just consider
non-negatively graded numerical semigroups. Notations, 
D--P stands for the Deligne--Pinkham's bound, NP for Pflueger's bound, and finally 
$\dim T^{1,+}$ for the dimension of the positive graded part of the first
cohomology module of the cotangent complex associated to $\k[\N]$, namely
$\dim T^{1,+}:=\sum_{s=1}^{\infty}\dim T^{1}(\k[\N])_{s}$, see theorem \ref{t1busc} below and
\cite{Bu80}.


\begin{table}[htb]
\caption{non-negatively graded semigroups of genus $\leq 6$}
\label{tab1}
\begin{tabular}{lcccc}
gaps& NP& $\dim\mathcal{M}_{g,1}^{\N}$& D--P & $\dim T^{1,+}$\\ \hline
1, 2, 4, 5, 8& 9 & 9 & 10 & 1\\ \hline
1, 2, 3, 5, 7& 10 & 10 & 11 & 1\\ \hline
1, 2, 3, 6, 7& 9 & 9 & 10 & 1\\ \hline
1, 2, 4, 5, 7, 10& 11 & 11 & 12 & 1\\ \hline
1, 2, 4, 5, 8, 11& 10 & 10 & 11 & 1\\ \hline
1, 2, 3, 5, 6, 9& 12 & 12 & 13 & 1\\ \hline
1, 2, 3, 5, 6, 10& 11 & 11 & 12 & 1\\ \hline
1, 2, 3, 5, 7, 9& 11 & 11 & 13 & 2\\ \hline
1, 2, 3, 5, 7, 11& 10 & 10 & 11 & 1\\ \hline
1, 2, 3, 6, 7, 11& 10 & 10 & 11 & 1\\ \hline
1, 2, 3, 4, 6, 8& 13 & 13 & 14 & 1\\ \hline
1, 2, 3, 4, 6, 9& 12 & 12 & 13 & 1\\ \hline
1, 2, 3, 4, 7, 8& 12 & 12 & 13 & 1\\ \hline
1, 2, 3, 4, 7, 9& 11 & 11 & 12 & 1\\ \hline
1, 2, 3, 4, 8, 9& 10 & 10 & 12 & 2\\
\end{tabular}
\end{table}

Let us now compare, see table \ref{tab2} below, the lower bound given by Pflueger with the upper bound
obtained in section four of this paper, cf. corollaries \ref{cor13} and \ref{boundfam2}. 
We also include the upper bound obtained by Contiero and Stoehr in \cite[Cor. 4.5]{c2013}
for the symmetric semigroup $\langle 6, 2+6\tau, 3+6\tau, 4+6\tau, 5+6\tau\rangle$ with
$\tau\geq 1$.

\begin{table}[htb]
\caption{$\dim\M$ for three families of semigroups}
\label{tab2}
\begin{tabular}{cccccc}
semigroup& NP & CFV-CS& D--P & $\dim T^{1,+}$\\\hline
$\langle 6, 3+6\tau, 4+6\tau, 7+6\tau, 8+6\tau\rangle$ & $8\tau+7$ & $8\tau+7$& $12\tau+5$ & $4\tau-2$\\ \hline
$\langle 6, 7+6\tau, 8+6\tau, 9+6\tau, 10+6\tau\rangle$ & $8\tau+11$ & $8\tau+11$& $12\tau+11$ & $4\tau$\\ \hline
$\langle 6, 2+6\tau, 3+6\tau, 4+6\tau, 5+6\tau\rangle$ & $8\tau+5$ & $8\tau+5$&$12\tau+1$ &  $4\tau-4$\\
\end{tabular}
\end{table}

We summarize the dimensions of $\M$ which appears in table \ref{tab1} and table \ref{tab2}, 
and also the dimension of $\M$ given by the theorem due Rim and Vitulli on negatively graded semigroups, in the next corollary.

\begin{cor} For each numerical semigroup $\N$ of genus $g\leq 6$, or any negatively graded numerical semigroup $\N$, 
or one of the following symmetric semigroups \linebreak
$\langle 6, 3+6\tau, 4+6\tau, 7+6\tau, 8+6\tau\rangle$, $\langle 6, 7+6\tau, 8+6\tau, 9+6\tau, 10+6\tau\rangle$
or $\langle 6, 2+6\tau, 3+6\tau, 4+6\tau, 5+6\tau\rangle$, we get
$$3g-2-\mathrm{ewt}(\N)=\dim\M=2g-2+\lambda(\N)-\dim T^{1,+}(\k[\N]).$$
\end{cor}

\begin{rmk}
We also recall that for the following two symmetric semigroups \linebreak $\N=<6,7,8>$ and $\N=<6,7,15>$, Pflueger's bound
does not provide the exact\- dimension of the $\M$, but one can see that $\dim\M=2g-2+\lambda(\N)-\dim T^{1,+}(\k[\N]).$
for this two particular semigroups.
 \end{rmk}
 
 

The dimension of the homogeneous part of degree $\l$ of the cotangent complex $T^{1}(\k[\N])$ can be
easily computed using a description of the cotangent complex given by Buchsweitz in \cite{Bu80}.
Let $\N:=\langle a_1,\dots,a_r\rangle$ be a numerical semigroup of genus $g>1$. By a theorem
due to Herzog, the ideal of $$C_{\N}:=\{(t^{a_1},\dots,t^{a_r})\,;\,t\in\k\}\subset\mathbb{A}^{r}$$
can be generated by isobaric polynomials $F_i$ which are differences of two monomials
$$F_i:=X_{1}^{\alpha_{i1}}\dots X_{r}^{\alpha_{ir}}-X_{1}^{\beta_{i1}}\dots X_{r}^{\beta_{ir}}$$
with $\alpha_i\cdot\beta_i=0$. As usual, the weight of $F_i$ is $d_i:=\sum_j a_j\alpha_{ij}=\sum_j a_j\beta_{ij}$.
For each $i$ let $v_i:=(\alpha_{i1}-\beta_{i1},\dots,\alpha_{ir}-\beta_{ir})$
be a vector in $\k^{r}$. 

\begin{thm}[cf. Thm. 2.2.1 of \cite{Bu80}]\label{t1busc}
For each $\l\notin\mathrm{End}(\N)$, 
$$\dim T^{1}(\k[\N])_{\l}=\#\{i\in\{1,\dots,r\}\,;\,a_i+\l\notin\N\}-\dim V_{\l}-1$$
where $V_{\l}$ is the subvector space of $\k^{r}$ generated by the vectors $v_i$
such that $d_i+\l\notin\N$. It also true that
$$\dim T^{1}(\k[\N])_s=0,\ \forall\,s\in\mathrm{End}(\N).$$
 \end{thm}

\noindent The following question was shared (in private communications) with a large number of specialist 
in the field of deformation theory and moduli of curves. We do not know a 
partial answer or even an example where the inequality fails.

\begin{question}
For which numerical semigroups it is true that
$$\dim \mathcal{M}_{g,1}^{\N}\leq 2g-2+\lambda-\dim T^{1,+}(\k[\N])\,?$$
\end{question}

\noindent The last result of this paper shows that Pflueger's lower bound can not be
greater than $2g-2+\lambda(\N)-\dim T^{1,+}(\N)$.

 \begin{lem}
 For any numerical semigroup $\N$ of genus $g\geq 1$,
 $$3g-2-\ewt(\N)\leq 2g-2+\lambda(\N)-\dim T^{1,+}(\N).$$
 \end{lem}
 \begin{proof}
For each $\l\in\mathbb{Z}$, set $A_{\l}:=\{i\in\{1,\dots,r\}\,;\,i+\l\notin\N\}$.
Using the theorem \ref{t1busc}, we obtain
\begin{equation*}
\dim T^{1,+}(\N)=\sum_{\ell \notin\End(\N)}(\# A_{\l}-\dim_{\k}V_{\l})-g+\lambda(\N).
\end{equation*} 
Hence, we just have to prove that
$\ewt(\N)-\sum_{\ell \notin\End(\N)}\sharp A_{\l}\geq 0$. We proceed
by induction on the genus $g$ of $\N$. The statement is trivial for $g=1$. If $\N$ is a numerical semigroup 
of genus $g>1$, whose biggest gap is $\l_{g}$, then consider the numerical semigroup 
$\N':=\N\cup\{\l_{g}\}$, whose genus is $g-1\geq 1$. It is clear that
$\{\l\notin\End(\N)\}=\{\l\notin\End(\N')\}\coprod\{\l\,\vert\,\l+a_i=\l_{g}\text{ and }\l+a_j\in\N,\,\forall\,j\neq i\}$.
Now the result follows easily.
\end{proof}


\begin{thebibliography}{99}
 
 
 \bibitem{ACGH85} {E.\ Arbarello, M.\ Cornalba, P.A.\ Griffiths
 and J.\ Harris},
 `Geometry of algebraic curves',
 {\em Grund\-lehren der Mathematischen Wissenschaften }267
 (Springer-Verlag, New York, 1985).
 
 \bibitem{BE77} {D.A.\ Buchsbaum and D.\ Eisenbud},
 `Algebra structures for finite free resolutions, and some structure
 teorems for ideals of codimension $3$',
 {\em Amer.\ J.\ Math.\ }99 (1977) 447--485.
 
 \bibitem{Bu80} {R.-O.\ Buchweitz},
 `On deformations of monomial curves',
 {\em Lecture notes in Mathematics }777 (1980) 205--220.
 
 \bibitem{CF}{A.\ Contiero, A. Fontes}, `On the locus of curves with 
 an odd subcanonical marked point', {\em arXiv:1804.09797} (2018).
 
 \bibitem{c2013} {A.\ Contiero, K.-O.\ Stohr}, 
 'Upper bounds for the dimension of moduli spaces of curves with 
 symmetric Weierstrass semigroups',
  {\em J. London Math. Soc.} 88 (2013) 580-598.
 
 \bibitem{Del73} {P.\ Deligne},
 `Intersections sur les surfaces r\'eguli\`eres' (SGA 7,
 Expos\'e X),
 {\em Lectures notes in Mathematics }340 (1973) 1--37.
 
\bibitem{EH87} {D.\ Eisenbud and J.\ Harris},
 `Existence, decomposition, and limits of certain Weierstrass points',
 {\em Invent.\ Math.\ }87 (1987) 495--515.
 
 \bibitem{LS} S. Lichtenbaum e M. Schlessinger, 
{`The Cotangent Complex of a Morphism'}, 
{\em Trans. Am. Math. Soc.},128 (1967), 41--70.
  
\bibitem{O91} {G.\ Oliveira},
 `Weierstrass semigroups and the canonical ideal of nontrigonal curves',
 {\em Manu\-scripta Math.\ }71 (1991) 431--450.
 
\bibitem{N16}{N.\ Pflueger},`On non-primitive Weierstrass points',
 {\em arXiv preprint}, arXiv:1608.05666 (2016).
 
\bibitem{NP2}{N.\ Pflueger},`Weierstrass semigroups on Castelnuovo curves',
  {\em arXiv preprint}, arXiv:1608.08178v1 (2016).
 
\bibitem{Pi74} {H.\ Pinkham},
 `Deformations of algebraic varieties with $G\sb{m}$-action',
 {\em Ast\'erisque 20 }(1974).
 
\bibitem{Rim}{D. Rim}
 `Formal deformation theory' (SGA 7,Expos\'e VII),
 {\em Lectures notes in Mathematics }288(1972) 32--132. 
 
\bibitem{RV77} {D.S.\ Rim and M.A.\ Vitulli},
 `Weierstrass points and monomial curves',
 {\em J.\ Algebra 48 }(1977) 454--476.
 
\bibitem{St93} {K.-O.\ St{\"o}hr},
 `On the moduli spaces of {G}orenstein curves with symmetric
 Weierstrass semigroups',
 {\em J.\ reine angew.\ Math.\ }441 (1993) 189--213.

\bibitem{W79} {R.\ Waldi},
`Deformation von Gorenstein-Singularit\"aten der Kodimension 3',
{\em Math.\ Ann.\ }242 (1979) 201--208.
 
\bibitem{W80} {R.\ Waldi},
`\"Aquivariante Deformation monomialer Kurven',
{\em Regensburger Math.\ Schriften }4 (1980) 1--88.
 
 
 \end{thebibliography}
 \end{document}